\documentclass[amsthm]{article}

\usepackage{amssymb}
\usepackage{amsthm}
\usepackage{mathtools}
\usepackage{tikz}
\usepackage{epigraph}

\theoremstyle{definition} 
     
\newtheorem{lemma}{Lemma}

\newtheorem{remark}{Remark}
\newtheorem{definition}{Definition}
\newtheorem{example}{Example}

\newcommand{\RR}{\mathbb R}
\newcommand{\ZZ}{\mathbb Z}

\title{The number $\pi$ and summation by $SL(2,\mathbb Z)$}

\author{Nikita Kalinin\footnote{CINVESTAV, Mexico. Research is supported by the grant 168647 of the Swiss National Science Foundation.}, Mikhail Shkolnikov\footnote{University of Geneva, Switzerland. Research is supported in part the grant 159240 of the Swiss National Science
Foundation as well as by the National Center of Competence in Research
SwissMAP of the Swiss National Science Foundation.}}

\begin{document}
\maketitle

\medskip
The sum (resp. the sum of the squares) of the defects in the triangle inequalities for the area one lattice parallelograms in the first quadrant has a surprisingly simple expression.

Namely, let $f(a,b,c,d)=\sqrt{a^2+b^2}+\sqrt{c^2+d^2}-\sqrt{(a+c)^2+(b+d)^2}$. 
Then,
$$\sum f(a,b,c,d) = 2,$$
$$\sum f(a,b,c,d)^2 = 2-\pi/2,$$
where the sum runs by all $a,b,c,d\in\mathbb Z_{\geq 0}$ such that $ad-bc=1$.

This paper is devoted to the proof of these formulae. We also discuss possible directions in study of this phenomena.  

\section{History: geometric approach to $\pi$}

\setlength{\epigraphwidth}{.48\textwidth}
\epigraph{What good your beautiful proof on the transcendence of $\pi$: why investigate such problems, given that irrational numbers do not even exists?}{Apocryphally attributed to Leopold Kronecker by Ferdinand Lindemann}

Digit computation of $\pi$, probably, is one of the oldest research directions in mathematics. Due to Archimedes we may consider the inscribed and superscribed equilateral polygons for the unit circle. Let $p_n$ (resp., $P_n$) be the perimeter of such an inscribed (resp., superscribed) $3\cdot2^n$-gon. The sequences $\{p_n\},\{P_n\}$ obey the recurrence $$P_{n+1}=\frac{2p_nP_n}{p_n+P_n}, p_{n+1}=\sqrt{p_nP_{n+1}}$$
and both converge to $2\pi$. However this gives no closed formula.

One of the major breakthrough in studying of $\pi$ was made by Euler, Swiss-born (Basel) German-Russian mathematician. In 1735, in his Saint-Petersburg Academy of Science paper, he calculated ({\it literally}) the right hand side of
\begin{equation}
\label{eq_euler}
\sum_{n=1}^{\infty} \frac{1}{n^2} = \frac{\pi^2}{6}.
\end{equation}

Euler's idea was to use the identity $$1-\frac{z}{6}+\dots=\frac{\sin(z)}{z}=\prod_{n=1}^\infty (1-\frac{z}{n^2 \pi^2}),$$
where the first equality is the Taylor series and the second equality happens because these two functions have the same set of zeroes. Equating the coefficient behind $z$ we get \eqref{eq_euler}. This reasoning was not justified until Weierstrass, but there appeared many other proofs. A nice exercise to get \eqref{eq_euler} is by considering the residues of $\frac{\cot(\pi z)}{z^2}$.

We would like to mention here a rather elementary geometric proof of \eqref{eq_euler} which is contained in [Cauchy, Cours d'Analyse, 1821, Note VIII].

\medskip

\begin{minipage}{0.3\textwidth}
\begin{tikzpicture}[scale=2]
\draw (1,0) arc (0:90:1);
\draw (0,0)--(1,0)--(1,0.53)--cycle;
\draw (1,0)--(0.88,0.47)--(0.88,0);
\draw (0,0)--(0.88,0.47)--(0.68,0.91) --cycle;
\draw (0.3,0.08) node{$\alpha$};
\draw (0.29,0.24) node{$\alpha$};
\end{tikzpicture}
\end{minipage}
\begin{minipage}{0.6\textwidth}
Let $\alpha=\frac{\pi}{2m+1}$. Let us triangulate the disk as shown in the picture. Then $\alpha$, the area of each segment, is bound by $\sin\alpha$ and $\operatorname{tg} \alpha$. Therefore $\cot^2\alpha\leq \frac{1}{\alpha^2}\leq \csc^2\alpha$. Writing $\frac{\sin ((2m+1)x)}{(\sin x)^{2m+1}}$ as a polynomial in $\cot x$ and using the fact that $\frac{\pi r}{2m+1}$ are the roots of this polynomial, through Vieta's Theorem we can find the sum of $\cot^2\alpha$ and $\csc^2\alpha$ for $\alpha=\frac{\pi r}{2m+1}, r=1,\dots, m$.
\end{minipage}

\medskip

 So, the above  geometric consideration gives a two-sided estimate for $\frac{1}{\pi^2}\sum_{n=1}^{m} \frac{1}{n^2}$ whose both sides converge to $\frac{1}{6}$ as $m\to\infty$.

\section{$SL(2,\ZZ)$-way to cut corners}
Recall that $SL(2,\ZZ)$ is the set of matrices 
$\begin{pmatrix} a & b \\ c& d \end{pmatrix}$ with $a,b,c,d\in\ZZ$ and $ad-bc=1$. With respect to matrix multiplication, $SL(2,\ZZ)$ is a group. We may identify such a matrix with the pair $(a,b),(c,d)\in\ZZ^2$ of lattice vectors such that the area of the parallelogram spanned by them is one.

\begin{definition}
A vector $v\in\ZZ^2$ is {\it primitive} if its coordinates are coprime. A polygon $P\subset \RR^2$ is called unimodular if 
\begin{itemize} 
\item the sides of $P$ have rational slopes; 
\item two primitive vectors in the directions of every pair of adjacent sides of $P$ give a basis of $\ZZ^2.$
\end{itemize}
\end{definition}
Note that a polygon's property of being unimodular is $SL(2,\ZZ)$-invariant. 

\begin{example}
The polygons $P_0$ and $P_1$ in Figure \ref{fig_octagon} are unimodular.
\end{example}

Let $P_0=[-1,1]^2$ and $D^2$ be the unit disk inscribed in $P_0$, Figure~\ref{fig_octagon}, left. Cutting all corners of $P_0$ by tangent lines to $D^2$ in the directions $(\pm 1,\pm 1)$  results in the octagon $P_1$ in which $D^2$ is inscribed, Figure \ref{fig_octagon}, right.

\begin{remark}
Note that if we cut a corner of $P_0$ by any other tangent line to $D^2$, then the resulting $5$-gon would not be unimodular. 
\end{remark}

\begin{definition}
For $n\geq 0$, the unimodular polygon $P_{n+1}$ circumscribing $D^2$ is defined to be the result of cutting all $4(n+1)$ corners of $P_n$ by tangent lines to $D^2$ in such a way that $P_{n+1}$ is a unimodular polygon.
\end{definition}

Note that passing to $P_{n+1}$ is unambiguous, because each unimodular corner of $P_n$ is $SL(2,\ZZ)$-equivalent to a corner of $P_0$ and the only possibility to unimodularly cut a corner at the point $(1,1)\in P_0$ is to use the tangent line to $D^2$ of the direction $(-1,1)$.

\begin{example}
The primitive vector $(1,1)$ is orthogonal to a side $S$ of $P_1$, belongs to the positive quadrant, and goes outside $P_1$. Two vectors orthogonal to the neighboring to $S$ sides of $P_2$ are $(2,1)$ and $(1,2)$.
 \end{example}

Let $Q$ be a corner of $P_n$. Let $v_1$ and $v_2$ be the primitive vectors orthogonal to the sides of $P_n$ at $Q$, pointing outwards. Then this corner is cut by the new side of $P_{n+1}$ orthogonal to the direction $v_1+v_2.$ Thus, we start with four vectors $(1,0),(0,1),(-1,0),(0,-1)$ --- the outward directions for the sides of $P_0$. To pass from $P_n$ to $P_{n+1}$ we order by angle all primitive vectors orthogonal to the side of $P_n$ and for each two neighbor vectors $v_1,v_2$ we cut the corresponding corner of $P_n$ by the tangent line to $D^2$, orthogonal to $v_1+v_2$. 
In particular, every tangent to $D^2$ line with rational slope contains a side of $P_n$ for $n$ large enough.

We can reformulate the above observation as follows:
\begin{lemma}
\label{lemma_all}
For all $a,b,c,d\in \ZZ_{\geq 0}$ with $ad-bc=1,$ such that $(a,b),(c,d)$ belong to the same quadrant, there is a corner of $P_n$ for some $n\geq 0$ supported by the primitive vectors $(a,b)$ and $(c,d).$ In $P_{n+1}$ this corner is cropped by the line orthogonal to $(a+c,b+d)$ and tangent to $D^2.$ 
\end{lemma}
The following lemma can be proven by direct computation.

\begin{lemma}
\label{lemma_cropped} 
In the above notation, the area of the cropped triangle is ${1\over 2}f(a,b,c,d)^2.$
\end{lemma}

We are going to prove that taking the limits of the lattice perimeters and areas of $P_n$ produces our formulae in the abstract. The next lemma is obvious.

\begin{lemma}
$\lim_{n\rightarrow\infty}{\operatorname{Area}}(P_n)={\operatorname{Area}}(D^2)$, $\lim_{n\rightarrow\infty}{\operatorname{Perimeter}}(P_n)=2\pi$.
\label{lem_convpn}
\end{lemma}

\begin{figure}
\begin{tikzpicture}
\draw(0,0)node{$\ $};

\begin{scope}[xshift=30]
\draw (2,2) circle (2cm);
\draw(0,0)--++(4,0)--++(0,4)--++(-4,0)--++(0,-4);
\end{scope}

\begin{scope}[xshift=200]
\draw (2,2) circle (2cm);
\draw(0,2.82842)--++(4-2.82842,4-2.82842)--++(2*2.82842-4,0)--++(4-2.82842,-4+2.82842)--++(0,-2*2.82842+4)--++(-4+2.82842,-4+2.82842)--++(-2*2.82842+4,0)--++(-4+2.82842,4-2.82842)--++(0,2*2.82842-4);
\end{scope}

\end{tikzpicture}
\caption{The disc is inscribed in the square $P_0$. Then, $P_1$ is the only unimodular octagon circumscribing $D^2$ which can be obtained by corner cuts of $P_0.$}
\label{fig_octagon}
\end{figure}
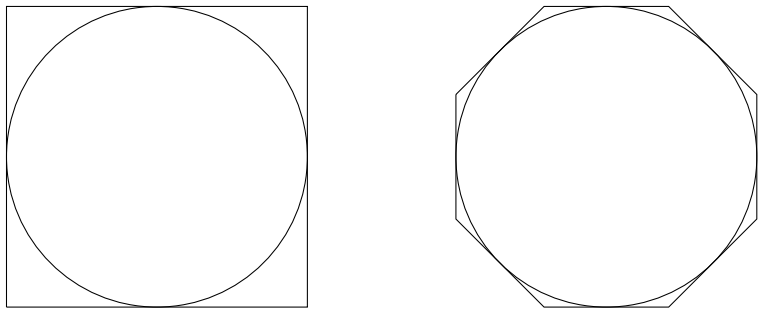

\section{Proofs}

The area of the intersection of $P_0\setminus D^2$ with the first quadrant is $1-\frac{\pi}{4}$. Therefore, it follows from Lemmata~\ref{lem_convpn},~\ref{lemma_cropped} that  $$\sum_{a,b,c,d} {1\over 2}f(a,b,c,d)^2 =1-\frac{\pi}{4},$$
which proves the second formula in the abstract.

\begin{definition}
Let $v$ be a primitive vector. We define the lattice length of a vector $kv, k\in\RR_{\geq 0}$ to be $k$.
\end{definition}
In other words, the length is normalized in each direction in such a way that all primitive vectors have length one. Note that the lattice length is $SL(2,\ZZ)$-invariant. 

The lattice perimeter of $P_n$ is the sum of the lattice lengths of its sides. For example, the usual perimeter of the octagon $P_1$ is $8\sqrt 2-4$ and the lattice perimeter is $2\sqrt 2+4.$ 

\begin{lemma}
\label{lemma_perim}
The lattice perimeter of $P_n$ 
\begin{itemize}
\item tends to zero as $n\to\infty$;
\item is given by $4\big(2-\sum f(a,b,c,d)\big),$ where the sum runs over $a,b,c,d\in\ZZ_{\geq 0}, ad-bc =1$, $(a,b)$ and $(c,d)$ are orthogonal to a pair of neighbor sides of some $P_k$ with $k\leq n$.
\end{itemize}
\end{lemma}
\begin{proof}
The second statement follows from the cropping procedure. To prove the first statement we note that for each primitive direction $v$ the length of the side of $P_n$, parallel to $v$, tends to $0$ as $n\to\infty$. The usual perimeter of $P_n$ is bounded (and tends to $2\pi$), and in the definition of the lattice length we divide by the lengths $|v|$ of the primitive directions $v$ for the sides of $P_n$. 

Therefore, for each $N>0$, the sum of the lattice lengths of the sides of $P_n$ parallel to $v$ with $|v|<N$ tends to zero, and the rest part of the lattice perimeter of $P_n$ is less than $\frac{2\pi}{N}$, which concludes the proof by letting $N\to\infty$.
\end{proof}
Finally, we deduce the first equality in the abstract from Lemmata~\ref{lemma_all},~\ref{lemma_perim}.

\section{Questions}
One may ask what happens for other powers of $f(a,b,c,d).$ There is a partial answer in degree $3$, which also reveals the source of our formulae.
 
For every primitive vector $w$ consider a tangent line to $D^2$ consisting of all points $p$ satisfying $w\cdot p+|w|=0.$ Consider a piecewise linear function $F:D^2\rightarrow \RR$ given by 
\begin{equation}
\label{eq_f}
F(p)=\inf_{w\in \ZZ^2\backslash{0}}(w\cdot p+|w|).\end{equation}

Performing {\it verbatim} the analysis of cropped tetrahedra applied to the graph of $F$ one can prove the following lemma.

\begin{lemma}
$4-2\sum f(a,b,c,d)^3=3\int_{D^2} F.$
\end{lemma}

 \begin{figure}[h]
\begin{minipage}{0.6\textwidth}
\includegraphics[width=\linewidth]{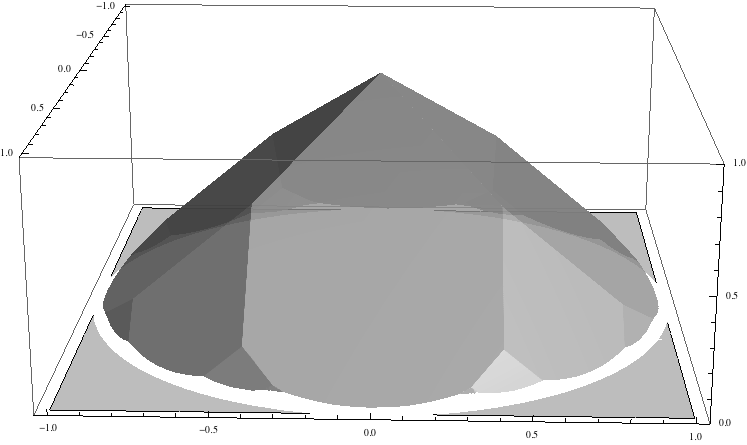}
\end{minipage}
\hfill%
\begin{minipage}{0.37\textwidth}\raggedleft
\includegraphics[width=\linewidth]{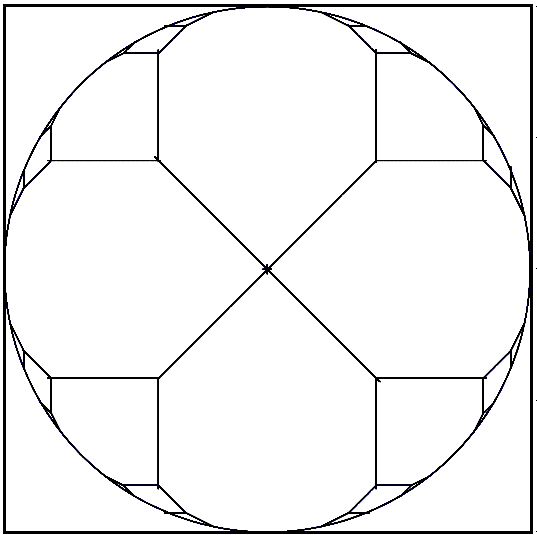}
\end{minipage}
\caption{The plot of $F$ and its corner locus (tropical analytic curve) $C$ for a disc.}
\label{fig_circlecurve}
\end{figure}

Now we describe the general idea behind the formulas.
 Denote by $C\subset D^\circ$ the locus of all points $p$ where the function $F$ is not smooth. The set $C$ is a locally finite tree (see Figure \ref{fig_circlecurve}). In fact, it is naturally a tropical curve (see \cite{us,announce}). The numbers $f(a,b,c,d)$ represent the values of $F$ at the vertices of $C$ and can be computed from the equations of tangent lines. 

Below we list some direction which we find interesting to explore. 
 
{\bf Coordinates on the space of compact convex domains.} For every compact convex domain $\Omega$ we can define $F_\Omega$ as the infimum of all support functions with integral slopes, exactly as in \eqref{eq_f}. Consider the values of $F_{\Omega}$ at the vertices of $C_\Omega$, the corner locus of $F_\Omega$. These values are the complete coordinates on the set of convex domains, therefore the characteristics of $\Omega$, for example, the area, can be potentially expressed in terms of these values. How to relate these coordinates of $\Omega$ with those of the dual domain $\Omega^*$?  

{\bf Higher dimensions.} We failed to reproduce this line of arguments ``by cropping'' for three-dimensional bodies, but it seems that we need to sum up by all quadruples of vectors $v_1,v_2,v_3,v_4$ such that $\mathrm{ConvHull}(0,v_1,v_2,v_3,v_4)$ contains no lattice points. 

{\bf Zeta function.} We may consider the sum $\sum f(a,b,c,d)^\alpha$ as an analog of the Riemann zeta function. This motivates a bunch of questions. What is the minimal $\alpha$ such that this sum converges? We can prove that $\frac{2}{3}<\alpha_{min}\leq 1$. This problem boils down to evaluating the sum $\sum \frac{1}{(|v||w||v+w|)^\alpha}$ by all pairs of primitive lattice vectors $v,w$ in the first quadrant such that the area of the parallelogram spanned by them is one. Can we extend this function for complex values of $\alpha$?

{\bf Other proofs.} It would be nice to reprove our formulae with other methods which are used to prove \eqref{eq_euler}. Note that the vectors $(a,b),(c,d)$ can be uniquely reconstructed by the vector $(a+c,b+d)$ and our construction reminds the Farey sequence a lot. Can we interpret $f(a,b,c,d)$ as a residue of a certain function at $(a+b)+(c+d)i$? The Riemann zeta function is related to integer numbers, could it be that $f$ is related to the Gauss integer numbers?

{\bf Modular forms.} We can extend $f$ to the whole $SL(2,\ZZ)$. If both vectors $(a,b),(c,d)$ belong to the same quadrant, we use the same definition. For $(a,b),(c,d)$ from different quadrant we could define $$f(a,b,c,d)=\sqrt{a^2+b^2}+\sqrt{c^2+d^2}-\sqrt{(a-c)^2+(b-d)^2}.$$ Then $$\sum\limits_{m\in SL(2,\ZZ)}f(m) = \sum\limits_{\substack {a,b,c,d\in \ZZ \\ ad-bc=1}}f(a,b,c,d)$$ is well defined. Can we naturally extend this function to the $\mathbb C/SL(2,\ZZ)$? Can we make similar series for other lattices or tessellations of the plane?





\subsection{Aknowledgement} 
 We would like to thank an anonymous referee for the idea to discuss the Euler formula,  
 and also Fedor Petrov and Pavol \v Severa for fruitful discussions. We want to thank the God and the universe for these beautiful formulae.


\begin{thebibliography}{1}

\bibitem{us}
N.~Kalinin and M.~Shkolnikov.
\newblock Tropical curves in sandpile models (in preparation).
\newblock {\em arXiv:1502.06284}, 2015.

\bibitem{announce}
N.~Kalinin and M.~Shkolnikov.
\newblock Tropical curves in sandpiles.
\newblock {\em Comptes Rendus Mathematique}, 354(2):125--130, 2016.

\end{thebibliography}


\end{document}